\nonstopmode

\documentclass[12pt]{article}

\usepackage{amsmath,amssymb,amsthm}
\usepackage{graphics,epsfig,calc}
\usepackage{amsmath}
\usepackage{latexsym,epsfig,bm,amssymb}
\usepackage{color}
\usepackage{amsthm,mathrsfs}

\usepackage{mathptmx}

\DeclareSymbolFont{AMSb}{U}{msb}{m}{n}
\DeclareSymbolFontAlphabet{\mathbb}{AMSb}

\newcommand{\beqn}{\begin{eqnarray}}
\newcommand{\eeqn}{\end{eqnarray}}
\newcommand{\be}{\begin{equation}}
\newcommand{\ee}{\end{equation}}
\newcommand{\ba}{\begin{array}}
\newcommand{\ea}{\end{array}}

\newcommand{\cF}{{\cal F}}

\newcommand{\cK}{{\cal K}}
\newcommand{\cM}{{\cal M}}

\newcommand{\cS}{{\cal S}}

\newcommand{\cV}{{\cal V}}
\newcommand{\cW}{{\cal W}}

\newcommand{\cX}{{\cal X}}

\newcommand{\al}{\alpha}

\newcommand{\ci}{\cite}
\newcommand{\de}{\delta}
\newcommand{\De}{\Delta}
\newcommand{\ds}{\displaystyle}
\newcommand{\fr}{\frac}
\newcommand{\ga}{\gamma}

\newcommand{\la}{\label}

\newcommand{\Lam}{\Lambda}
\newcommand{\na}{\nabla}

\newcommand{\om}{\omega}
\newcommand{\vp}{\varphi}

\newcommand{\ov}{\overline}

\newcommand{\pa}{\partial}
\newcommand{\re}{\ref}

\newcommand{\Si}{\Sigma}
\newcommand{\si}{\sigma}
\newcommand{\ti}{\tilde}

\newcommand{\ve}{\varepsilon}

\newcommand\C{{\mathbb C}}
\newcommand\R{{\mathbb R}}
\newcommand\N{{\mathbb N}}
\newcommand\Z{{\mathbb Z}}
\newcommand\T{{\mathbb T}}

\newcommand{\Ga}{\Gamma}

\newcommand{\vka}{\varkappa}

\newcommand{\cm}{{\rm m}}
\newcommand\nab{{\bf \nabla}}

\newcommand{\5}{{\hspace{0.5mm}}}

\newcommand{\3}{{\hspace{0.2mm}}}

\newcommand{\rRe}{{\rm Re\5}}
\newcommand{\rIm}{{\rm Im\5}}

 \newcommand{\st}{\stackrel} 
 
 \newcommand{\toLtt}{\st{L^2(\ov\T)}{-\!\!\!-\!\!\!-\!\!\!\!\longrightarrow}}
\newcommand{\toLt}{\st{L^2(\T)}{-\!\!\!-\!\!\!-\!\!\!\!\longrightarrow}} 
\newcommand{\toLwt}{\st{L^2_w({\ov\T})}{-\!\!\!-\!\!\!-\!\!\!\!\longrightarrow}}
 
\newcommand{\toLC}{\st{C(\T)}{-\!\!-\!\!\!\!\longrightarrow}} 
\newcommand{\tocX}{\st{\cX}{-\!\!\!\!\longrightarrow}}

\newcommand{\Ker}{{\rm Ker\5}}

\newcommand{\Spec}{{\rm Spec\3}}

\renewcommand{\theequation}{\thesection.\arabic{equation}}
\newtheorem{theorem}{Theorem}[section]
\renewcommand{\thetheorem}{\arabic{section}.\arabic{theorem}}
\newtheorem{definition}[theorem]{Definition}

\newtheorem{lemma}[theorem]{Lemma}
\newtheorem{example}[theorem]{Example}
\newtheorem{remark}[theorem]{Remark}
\newtheorem{remarks}[theorem]{Remarks}
\newtheorem{cor}[theorem]{Corollary}
\newtheorem{proposition}[theorem]{Proposition}

\newcommand{\bd}{\begin{definition}}
 \newcommand{\ed}{\end{definition}}
\newcommand{\bt}{\begin{theorem}}
 \newcommand{\et}{\end{theorem}}
\newcommand{\bqt}{\begin{qtheorem}}
 \newcommand{\eqt}{\end{qtheorem}}

\newcommand{\bp}{\begin{proposition}}
 \newcommand{\ep}{\end{proposition}}

\newcommand{\bl}{\begin{lemma}}
 \newcommand{\el}{\end{lemma}}
\newcommand{\bc}{\begin{cor}}
 \newcommand{\ec}{\end{cor}}

\newcommand{\bex}{\begin{example}}
 \newcommand{\eex}{\end{example}}
\newcommand{\bexs}{\begin{examples}}
 \newcommand{\eexs}{\end{examples}}

\newcommand{\bexe}{\begin{exercice}}
 \newcommand{\eexe}{\end{exercice}}

\newcommand{\br}{\begin{remark} }
 \newcommand{\er}{\end{remark}}
\newcommand{\brs}{\begin{remarks}}
 \newcommand{\ers}{\end{remarks}}

\begin{document}

\begin{titlepage}
\vspace{2cm}

\begin{center}
{\Large\bf 
 On  orbital stability of ground states for finite crystals 
\medskip\\
in fermionic Schr\"odinger--Poisson model
}
\end{center}
\bigskip\bigskip

 \begin{center}
{\large A. Komech}
\footnote{
Supported partly 
by 
Austrian Science Fund (FWF): P28152-N35,
and the grant of  RFBR 16-01-00100.}
\\
{\it Faculty of Mathematics of Vienna University\\
and Institute for Information Transmission Problems RAS } \\
e-mail:~alexander.komech@univie.ac.at
\bigskip\\
{\large E. Kopylova}
\footnote{
Supported partly by  
Austrian Science Fund (FWF): P27492-N25,
and the grant of RFBR 16-01-00100.}
\\
{\it Faculty of Mathematics of Vienna University\\
and Institute for Information Transmission Problems RAS} \\
 e-mail:~elena.kopylova@univie.ac.at
\end{center}
\vspace{1cm}

 \begin{abstract}
We consider the 
Schr\"odinger--Poisson--Newton equations for finite  crystals
under periodic boundary conditions with one ion per cell of a lattice.
The electron field is described by the $N$-particle Schr\"odinger equation with antisymmetric 
wave function.

Our main results are 
i) the global dynamics with  moving ions, and
ii) 
the  orbital stability of 
 periodic  ground state
under a novel  
Jellium and Wiener-type conditions on the ion charge density.
Under Jellium condition
 both ionic and  electronic charge densities  
 for the ground state 
 are uniform.

\end{abstract}
\bigskip

\bigskip


\end{titlepage}

\section{Introduction}
First mathematical results on the stability of matter were obtained 
by Dyson and Lenard 
in \ci{D1967, DL1968} 
where the energy bound from below
was established.
The thermodynamic limit
for the Coulomb systems  
was first studied by Lebowitz and Lieb 
\ci{LL1969,LL1973}, see the survey and further development in \ci{LS2010}.
These results were extended  by Catto, Le Bris,  Lions, 
and others to the Thomas--Fermi and Hartree--Fock models
\ci{CBL1998,CBL2001,CBL2002}. 
Further results in this direction are due to
Canc\'es,  Lahbabi, Lewin, Sabin, Stoltz, and others
 \ci{CLL2013,CS2012, BL2005, LS2014-1, LS2014-2}.
All these results concern  either
the convergence of
the ground state of finite particle systems
in the thermodynamic limit or 
the existence of the ground state
for infinite particle systems. 

However, no attention was paid to
the dynamical stability of crystals
with moving ions. 
This stability is 
necessary for a rigorous analysis 
of fundamental quantum phenomena in the solid state physics: 
heat conductivity, electric conductivity, thermoelectronic emission, photoelectric effect, 
Compton effect, 
etc., see \ci{BLR}.

In present paper we
 consider the coupled
Schr\"odinger--Poisson--Newton equations for finite  crystals
under periodic boundary conditions with one ion per cell of a lattice.
The electrons are described by the $N$-particle Schr\"odinger equation with antisymmetric 
wave function.
We
construct the global dynamics   of crystals
with  moving ions and prove the conservation of energy and charge.
 
Our main result
is  the  orbital stability of every 
 ground state with periodic arrangement of ions
under the novel `Jellium' and  Wiener-type  conditions on the ion charge density.

 The electron field is described by  the many-particle 
Schr\"odinger equation in the space of antisymmetric wave functions
which corresponds to the Pauli exclusion principle.
The ions are described as classical particles   
corresponding to the  
Born and Oppenheimer  approximation.
The ions interact with the electron field via 
the scalar potential, which  is a solution to the corresponding Poisson equation.
We find a novel 
stability criterion (\re{Wai}), (\re{W1}).
\medskip

We consider crystals which occupy the finite torus $\T:=\R^3/N\Z^3$
and have one ion per cell of the cubic lattice $\Ga:=\Z^3/N\Z^3$, where $N\in\N$.
The cubic
lattice  is chosen for the simplicity of notation.
We denote by  
$\sigma(x)$ the charge density of one 
ion,
\be\la{ro+}
\si\in C^2(\T),\qquad
\int_{\T} \sigma(x)dx=eZ>0, 
\ee 
where $e>0$ is the elementary charge.
Let us denote 
\be\la{Fv}
\ov\T:=
\T^{\ov N}
:=\{\ov x=(x_1,...,x_{\ov N}): x_j\in\T,\quad j=1,...,\ov N\},\qquad \ov N:=N^3.
\ee
\bd $\cF$ is the `fermionic' Hilbert space of  complex antisymmetric functions
 $\psi(x_1,...,x_{\ov N})$ on $\ov\T$
with the norm
\be\la{FN}
\Vert\psi\Vert_\cF^2:= \Vert\na^\otimes\psi\Vert_{L^2(\ov\T)}^2+\Vert\psi\Vert_{L^2(\ov\T)}^2,
\ee
where $\na^\otimes$ denotes the gradient with respect to $\ov x\in\ov\T$.
\ed
Let 
$\psi(\cdot,t)\in\cF$ for $t\in\R$  be the antisymmetric  
wave function of the fermionic electron field, 
$q(n,t)$ denotes the ion displacement  from the reference position $n\in\Ga$,
and
$\Phi(x,t)$ be the electrostatic  potential generated by the ions and electrons.
We assume $\hbar=c=\cm=1$, where $c$ is the speed of light and $\cm$ is the electron mass.
Let us denote the `second quantized' operators on $\cF$,
\be\la{2De}
\De^\otimes:=\sum_{j=1}^{\ov N} \De_{x_j};\qquad \Phi^\otimes(\ov x,t):=\sum_{j=1}^{\ov N} \Phi(x_j,t).
\ee
The coupled Schr\"odinger-Poisson-Newton equations  read as follows
\beqn\la{LPS1}
i\dot\psi(\ov x,t)\!\!&=&\!\!-\fr12\De^\otimes\psi(\ov x,t)-e\Phi^\otimes(\ov x,t)\psi(\ov x,t),\qquad \ov x\in\ov\T,
\\
\nonumber\\
-\De\Phi(x,t)\!\!&=&\!\!\rho(x,t):=\sum_{n\in\Ga}
\sigma(x-n-q(n,t))+\rho^e(x,t),\qquad x\in\T,
\la{LPS2}
\\
\nonumber\\
M\ddot q(n,t)
\!\!&=&\!\!-(\nab\Phi(x,t),\sigma(x-n-q(n,t))), 
\qquad n\in\Ga.
\la{LPS3}
\eeqn
Here
the 
brackets $(\cdot,\cdot)$
 stand for the  scalar product on the real Hilbert
space $L^2(\T)$ and for its different extensions,   
$M>0$ is the mass of one ion,
and the electronic charge density is defined by
\be\la{re}
\rho^e(x,t)
:=-e\int_{\ov\T}\sum_{j=1}^{\ov N} \de(x-x_j)|\psi(\ov x,t)|^2\,d\ov x,~~ x\in\T.
\ee
 Similar finite periodic approximations of crystals are treated in all textbooks on 
quantum theory of solid state
\ci{Born, Kit, Zim}. 
However,   the  stability 
of ground states in
this model was newer discussed.
 \medskip
 
 The total electronic charge (up to a factor) is defined  by
 \begin{equation}\la{Q}
 Q(\psi,q,p):= \int_{\ov\T}|\psi(\ov x)|^2d\ov x=\Vert\psi\Vert_{L^2(\ov\T)}^2.
 \end{equation}
The  Poisson equation (\re{LPS2}) implies that
\begin{equation}\la{r0}
\ds\int_{\T}\rho(x,t)dx=0.
\end{equation}
Hence, the potential $\Phi(x,t)$ can be eliminated
from  the system (\re{LPS1})--(\re{LPS3})
using 
the operator $G:=(-\De)^{-1}$,
see (\re{fs}) for a more precise definition. 
Then
the system  (\re{LPS1})--(\re{LPS3}) can be written in the Hamilton form
\be\la{HSi}
i\dot \psi(\ov x,t)=\fr12 \pa_{\ov\psi} E,\qquad\dot q(n,t)= \pa_{p(n)}E,\quad \dot p(n,t)=-\pa_{q(n)} E.
\ee
Here $\pa_{\ov\psi}:=\fr12 [\pa_{\psi_1}-i\pa_{\psi_2}]$, where $\psi_1:=\rRe\psi$ and $\psi_2:=\rIm\psi$, and 
the Hamilton functional (energy)
reads 
\be\la{Hfor}
  E(\psi, q, p)=\fr12\int_{\ov\T} |\na^\otimes\5\psi(\ov x)|^2\,d\ov x+
  \fr12 (\rho,G\rho)+
  \sum_{n\in\Ga} \fr{p^2(n)}{2M}.
\ee
Here  $ q:=(q(n): ~~n\in\Ga)\in\ov\T$,
 $ p:=(p(n):~~n\in\Ga)\in\R^{3\ov N}$, and
the total charge density
 $\rho(x)$ is the sum of
the ion and electronic charge densities,
\beqn\la{Hfor2}
\rho(x):=\rho^i(x)+\rho^e(x),\quad
\rho^i(x):=\sum\limits_{n\in\Ga}\si(x-n-q(n)),\qquad x\in\T,
\eeqn
in accordance with (\re{LPS2}) and (\re{re}).  
The identity (\re{r0}) implies 
the normalization
\be\la{rQ}
\Vert\psi(\cdot,t)\Vert_{L^2(\ov\T)}^2=Z,\qquad t\in\R.
\ee
We denote the Hilbert manifolds
\be\la{cM}
 \cV:=H^1(\ov\T)\otimes \ov\T\otimes \R^{3\ov N},\qquad
\cM:=\{X\in\cV: Q(X)=Z\}.
\ee
We prove the 
the global well-posedness of the dynamics:
for any $X(0)\in\cM$
there exists a unique
solution 
$X(t)\in C(\R,\cV)$ 
to (\re{HSi}), and
the energy and charge conservations  hold:
\be\la{EQ}
E(X(t))=E(X(0)),\quad Q(X(t))=Q(X(0)),\qquad t\in\R.
\ee
The charge conservation formally follows by
the Noether theory \ci{A, GSS87, KQ} 
due to the $U(1)$-invariance of the Hamilton functional:
\be\la{U1}
E(e^{i\al }\psi,q,p)=E(\psi,q,p),\qquad \al \in\R.
\ee

Our main goal is the stability of 
ground states, i.e.,
solutions to (\re{HSi})
with minimal energy  (\re{Hfor}).
We consider only ground states 
with  $\Ga$-periodic arrangement of ions (nonperiodic arrangements 
exist for some degenerate 
densities $\si$, 
see Remark \re{r1} iii) below).
\medskip

We impose two special Jellium and the Wiener conditions (\re{Wai}) and  (\re{W1}) 
onto the  ion densities $\si(x)$. 
The Wiener condition is a suitable 
version of the Fermi Golden Rule for crystals.
The Jellium condition implies that total density of ions 
is uniform when $q(n,t)\equiv 0$, see (\re{sipi}).
\medskip

The energy (\re{Hfor}) is nonnegative, and its minimum is zero.
We show  that under the Jellium condition 
 all
ground states  with $\Ga$-periodic arrangement of ions have the form 
\begin{equation}\la{gr}
S(t):=(\psi_0e^{-i\om_0t},\ov r,0), ~~~~~r\in\T.
\end{equation}
Here
\begin{equation}\la{gr2}
 \ov r\in \ov\T:\quad \ov r(n)=r,\,\,n\in\Ga,
\end{equation}
while $\psi_0$ is an eigenfunction 
\be\la{eig}
-\fr12 \De^\otimes\5\psi_0(\ov x)=\om_0\psi_0(\ov x),\qquad \ov x\in\ov\T,
\ee
corresponding to the minimal eigenvalue  
$
\om_0:=\min \5\5\Spec(-\fr12 \De^\otimes).
$
\medskip

We establish
the stability  
of the real  4-dimensional `solitary manifold'
\be\la{cS}
\cS=\{S_{\al,r}=(\psi_\al,\ov r,0):~\psi_\al(\ov x)
\equiv e^{i\al}\psi_0(\ov x),~~\al\in [0,2\pi];~~ r\in\T\},
\ee
where $\psi_0$ is a fixed eigenfunction, satisfying  the additional restriction (\re{adr}).
 The normalization (\re{rQ}) and the identity (\re{eig}) imply that
\be\la{ES}
E(S)= \om_0Z,\qquad S\in\cS.
\ee

Our main result is the following theorem.

\bt\la{tm}
Let the Jellium and Wiener conditions  (\re{Wai}) and (\re{W1}) hold as well as  (\re{adr}).
Then
for any 
$\ve>0$  there exists $\de=\de(\ve)>0$ such that for 
$X(0)\in\cM$ with
$d_\cV(X(0),\cS)<\de$ we have
\be\la{m}
d_\cV(X(t),\cS)<\ve,\qquad t\in\R,
\ee
where  $X\in C(\R,{\cal V})$
is the corresponding solution
 to (\ref{HSi}).

\et

This theorem means  the `orbital stability' in the sense of \ci{GSS87},
 since the manifold
$\cS=S^1\times\T$ is an orbit of the symmetry group
$U(1)\times \T$.
\medskip

Let us comment on our approach.
We prove the local well-posedness for the system  (\re{HSi})
by
the contraction mapping principle.
The global  well-posedness we deduce
from the  energy conservation which follows
by the Galerkin approximations. 
\medskip

The orbital stability of the solitary manifold $\cS$ is deduced
from  the lower energy estimate
\be\la{BLi}
E(X)-\om_0Z\ge \nu\,d^2(X,\cS)\qquad{\rm if}\qquad d(X,\cS)\le \de,\qquad X\in\cM,
\ee
where
$\nu,\de>0$ and `$d$' is the distance in the `energy norm'.
We deduce this estimate  from the positivity of the Hessian $E''(S)$
for $S\in \cS$ in the orthogonal directions to $\cS$ on the manifold $\cM$.
We show that the Wiener condition \eqref{W1} is necessary for this positivity under 
the Jellium condition \eqref{Wai}.
The last condition  cancels the negative
energy which is provided by the electrostatic instability
 (`Earnshaw's Theorem' \ci{Stratton}, 
see \ci[Remark 10.2]{KKpl2015}).
We expect that this  condition is also necessary for the positivity of $E''(S)$;
however, this is still an open challenging problem.
Anyway,
the positivity of $E''(S)$ can break
down when condition  \eqref{Wai}  fails. We have shown this in \ci[Lemma 10.1]{KKpl2015}
in the context of infinite crystals; the proof however extends directly to the finite
crystals.

\brs\la{r1}
{\rm
	i)
In the case of infinite crystal, corresponding to $N=\infty$, the orbital stability seems  impossible. 
 Namely, for  $N=\infty$ the estimates (\ref{s1}), (\ref{s2}), (\ref{GP2}) and (\ref{fp}) break down,
 as well as the estimate of type (\ref{BLi}) which is due to the discrete spectrum of the energy
 Hessian $E''(S)$ on the compact torus.
\medskip\\
	ii)  The identity (\re{roZ}) generically breaks down for the eigenfunctions 
(\re{gex})
if the  condition  (\re{adr}) fails. Respectively, the orbital stability 
of these `mixed states' is an open problem.

}
\ers

Let us comment on previous works in this field.
\medskip\\
 The ground state for crystals 
in the Schr\"odinger--Poisson model was constructed in 
\ci{K2014,K2015}; its linear stability was proved in \ci{KKpl2015}.  
\medskip\\
In the Hartree--Fock model
the crystal ground state 
 was constructed for the first time by Catto, Le Bris, and  Lions  \ci{CBL2001,CBL2002}.
For the Thomas--Fermi model, see \ci{CBL1998}.
\medskip\\
In \ci{CS2012}, Canc\'es and Stoltz have established the well-posedness  for 
the dynamics of  
local perturbations of the  ground state density matrix
in the  {\it random phase approximation}
for the reduced  Hartree--Fock equations
with the Coulomb  pairwise interaction potential $w(x-y)=1/|x-y|$.
However, the  space-periodic nuclear potential
in the equation \ci[(3)]{CS2012}
does not depend on time that corresponds to 
the fixed nuclei positions. 
 \medskip\\
The nonlinear Hartree--Fock dynamics
with the Coulomb potential and
without the  random phase approximation
was not studied previously,
see the discussion in 
\ci{BL2005} and in the Introductions of \ci{CLL2013,CS2012}.
\medskip\\
In \ci{CLL2013} 
E. Canc\`es, S. Lahbabi, and M. Lewin have considered the random 
reduced HF model of crystal  when 
the ions charge density and the electron density matrix are random processes,
and the action of the lattice translations on the probability space is ergodic.
The authors obtained suitable generalizations of the Hoffmann--Ostenhof 
and Lieb--Thirring inequalities  for ergodic density matrices, 
and
construct a random potential which is a solution  to 
 the Poisson equation 
with the corresponding stationary stochastic  charge density. 
The main result is the  coincidence of this model with the thermodynamic limit in  
the case of the short-range Yukawa interaction.
\medskip\\
In \ci{LS2014-1}, Lewin and Sabin have established the well-posedness for the 
reduced von Neumann equation, describing the Fermi gas,
with density matrices of infinite trace 
and pair-wise interaction potentials $w\in L^1(\R^3)$. Moreover, they  
proved the asymptotic stability of translation-invariant stationary states 
for 2D Fermi gas \ci{LS2014-2}.
\bigskip

The paper is organized as follows.
In Section 2 we introduce function spaces.
In Section 3 we collect all our assumptions.
In Section 4
we describe all fermionic jellium ground states and give basic examples.
In Section 5  we prove 
the stability of the solitary manifold $\cS$ establishing  
the positivity  the energy Hessian. 
In Appendices we construct the global dynamics.
\medskip\\
{\bf Acknowledgments.} The authors are grateful to Herbert Spohn for helpful 
discussions and remarks.


\setcounter{equation}{0}
\section{Function spaces and integral equation}
The operator $G:=(-\De)^{-1}$   is well defined in the Fourier series:
\be\la{fs}
\rho(x)=\sum_{\xi\in\Xi}\ti\rho(\xi)e^{i\xi x},\quad 
G\rho:=\sum_{\xi\in\Xi\setminus 0}\fr{\ti\rho(\xi)}{\xi^2}e^{i\xi x},\qquad x\in\T.
\ee
Here $\ti\rho(0,t)=0$ by (\re{r0}). 
 Hence,
$\Phi(\cdot,t)=G\rho(\cdot,t)$ up to an additive constant $C(t)$,
 which can be offset by a gauge transform 
$\psi(\ov x,t)\mapsto\psi(\ov x,t)\exp(ie\ds\int_0^t C(s)ds)$. 
Substituting $\Phi(\cdot,t)=G\rho(\cdot,t)$
into the remaining equations (\ref{LPS1}) and (\ref{LPS3}) we can write these equations as
\be\la{vf}
\dot X(t)=F(X(t)),\qquad t\in\R,
\ee
where $X(t)=(\psi(\cdot,t), q(\cdot,t), p(\cdot,t))$ with $p(\cdot,t):=\dot q(\cdot,t)$.
Equation (\ref{vf}) with the normalization (\ref{rQ}) 
is equivalent, up to a gauge transform,
to the
system (\ref{LPS1})--(\ref{LPS3}).
Finally, the equation (\re{vf}) 
can be written in the Hamilton form (\ref{HSi}), which is equivalent to 
\begin{equation}\la{HS}
\dot X(t)=JE'(X(t)),
\end{equation}
where 
\begin{equation}\la{HS2}
J=\left(
\begin{array}{rrr}
-i/2 &0&0\\
0   &0&1\\
0   &-1&0
\end{array}
\right).
\end{equation}

We will use the following function spaces with $s=0,\pm 1$.
Let us define the Sobolev space $H^s(\ov\T)$ as the real Hilbert space of 
complex-valued functions
with the  scalar product 
\be\la{sp}
(\psi,\vp)_s:=\rRe\int_{\ov\T}\sum_{|\al|\le s} \pa^\al\psi(\ov x)
\pa^\al\ov\vp(\ov x)d\ov x,\qquad s=0, 1.
\ee
By definition, $H^{-1}(\ov\T)$ is the dual space to $H^1(\ov\T)$,
which will be identified with distributions by means of the scalar product in $H^0(\ov\T)$.

\bd
i)  Let us denote the  real Hilbert space 
$\cW^s:=H^s(\ov\T)\oplus \R^{3\ov N}\oplus \R^{3\ov N}$ for $s=0,\pm1$.
\medskip\\
ii) $\cV^s:=
H^s(\ov\T)\otimes \ov\T\otimes \R^{3\ov N}$ is 
the real Hilbert manifold endowed with the metric 
\begin{equation}\la{dVs}
d_{\cV^s}(X,X'):=\Vert\psi-\psi'\Vert_{H^s({\T})}+|q-q'|+|p-p'|,\qquad X=(\psi,q,p),\quad
X'=(\psi',q',p')
\end{equation}
and with the `quasinorm'
\be\la{cVs}
|X|_{\cV^s}:=\Vert\psi\Vert_{H^s(\ov\T)}+|p|,\qquad X=(\psi,q,p).
\ee

\ed
The linear space $\cW^s$ is  the tangent space to 
the Hilbert manifold $\cV^s$ in each point $X\in\cV^s$. 
We will write $\cX:=\cV^0$, $\cV:=\cV^1$, $\cW:=\cW^1$,
and $(\cdot,\cdot)_0=(\cdot,\cdot)$,
which agrees with the definition of  
the scalar product on the real Hilbert space $L^2(\T)$. In particular,
\begin{equation}\la{1i}
(1,i)=0.
\end{equation}
Denote by the brackets $\langle\cdot,\cdot\rangle$
the scalar product on $\cX$ and also the duality between 
$\cW^{-1}$ and $\cW^1$:
\be\la{dWW}
\langle Y,Y'\rangle:=(\vp,\vp')+\vka\vka'+\pi\pi',\qquad Y=(\vp,\vka,\pi),\quad
Y'=(\vp',\vka',\pi').
\ee
Obviously,
\begin{equation}\la{EQV}
|X|_\cV^2\le C[E(X)+Q(X)],\qquad X\in\cV.
\end{equation}

We construct global dynamics for the system  (\re{HS}).
This system 
is a  nonlinear infinite-dimensional perturbation of the free Schr\"odinger equation.
We rewrite it in the integral  form
\begin{equation}\la{LPSi}
\left\{\begin{array}{lll}
\psi(t)&=&e^{\fr i2t\De^\otimes } \psi(0)+ie\ds\int_0^t e^{\fr i2(t-s)\De^\otimes} [ \Phi^\otimes(s)\psi(s) ]ds,\\
\\
q(n,t)&=&q(n,0)+\frac 1M\ds\int_0^t p(n,s)ds\mod N\Z^3,\\
\\
p(n,t)&=&p(n,0)-\ds\int_0^t (\nabla \Phi(s),\sigma(\cdot-n-q(n,s))) ds,
\end{array}\right|
\end{equation}
where $\Phi(s):=G\rho(s)$.
In the vector form (\ref{LPSi}) reads
\begin{equation}\la{LPSiv}
X(t)=e^{-tA}X(0)+\int_0^t  e^{-(t-s)A} N(X(s)) ds\mod \left(\ba{c}0\\ N\Z^3\\0\ea\right).
\end{equation}
Here
\beqn\la{HN}
A&=&\left(\begin{array}{ccc}
	 -\fr i2\De^\otimes  & 0 & 0\\
	0&0&0\\
	0&0&0
\end{array}\right),
\nonumber\\
\nonumber\\
N(X)&=&(ie \Phi^\otimes\5\psi ~, p,~f),\,\,\, f(n):=-(\nabla \Phi,\sigma(\cdot-n-q(n))),\,\,\,\Phi:=G\rho,
\eeqn
where $\rho$ is defined by (\re{Hfor2}).

\setcounter{equation}{0}
\section{Main assumptions}

Our main result concerns   the orbital stability of the
ground states (\re{gr}).
We will see that the ground states  can be stable  
depending on the choice of the ion density $\sigma$.
We study  densities 
$\sigma$ satisfying the following two conditions. 
First, we will assume 
\begin{equation}\la{Wai}
\mbox{\bf The Jellium Condition:}~~~~~ \ti\si(\xi)
  :=\int_\T e^{i\xi x}\si(x)dx
 =0,\quad \xi\in \ga^*\setminus 0,
~~~~~~~~~~~~ ~~~~~~~~~~~~~~~~~~ ~
\end{equation}
where $\ga^*:=2\pi\Z^3$. 
This condition implies that total density of ions 
is uniform when $q(n,t)\equiv 0$,
\begin{equation}\la{sipi}
 \sum_{n\in\Ga}\si(x-n)\equiv eZ,\qquad x\in\T.
 \end{equation}
 The simplest example of such a 
$\sigma$ is a constant over the unit cell of a given lattice, which is what physicists 
usually call Jellium \cite{GV2005}. Moreover,
this condition  holds 
for a broad class of  functions  $\si$,
see   Section \re{sex}. 
Here we study this model in the rigorous context of the Schr\"odinger--Poisson equations. 

Furthermore, we will assume the 
Wiener type
 spectral property 
\begin{equation}\la{W1}
\mbox{\bf The Wiener Condition:}~~~\Si(\theta):=\sum_{m\in\Z^3}\Big[
 \fr{\xi\otimes\xi}{|\xi|^2}|\ti\si(\xi)|^2\Big]_{\xi=\theta+2\pi m}>0,\
\quad \theta\in \Pi^*\setminus \ga^*,
\end{equation}
where the Brillouin zone $\Pi^*$ is defined by
\begin{equation}\la{PPG}
 \Pi^*:= \{\xi=(\xi^1,\xi^2,\xi^3)\in\Xi:0\le \xi^j\le 2\pi,~~j=1,2,3\},\quad\Xi:=\fr{2\pi}N\Z^3.
\end{equation}
This condition is
 an analogue of the Fermi Golden Rule
for crystals. 
It is independent of (\re{Wai}).
 We have introduced  
conditions of type (\re{Wai}) and  (\re{W1}) in \ci{KKpl2015} 
 in the framework of infinite crystals.
\br\la{rW}
{\rm
i) The series \eqref{W1} converges 
for $\theta\in\Xi\setminus \ga^*$
by the Parseval identity since $\si\in L^2(\T)$
by \eqref{ro+}.
\medskip\\
ii) The matrix $\Si(\theta)$ is $\ga^*$-periodic outside  $\ga^*$.
Thus, (\re{W1}) means that
$\Si(\theta)$ is a positive matrix
for $\theta\in \ov\Pi^*\setminus 0$, where $\ov\Pi^*$
is
 the `discrete torus' $\Xi/\ga^*$.
 }
\er

The series \eqref{W1} is a nonnegative matrix.
Hence,
 the Wiener condition holds `generically'.
For example 
it holds if
\begin{equation}\la{W1s}
\ti\si(\xi)\ne 0,\qquad \xi\in \Xi\setminus\ga^*,
\end{equation}
i.e., (\re{Wai}) are the only zeros of $\ti\si(\xi)$.
However,  (\re{W1}) does not hold 
for the simplest Jellium model,
 when $\sigma$ is constant on the unit cell, see (\re{sic}) and (\re{sJM}). 
\medskip

Finally, we need an additional condition for the orbital stability 
of the ground state (\re{gr}).
Namely, every eigenfunction  (\re{eig}) admits an expansion
in exterior products (see (\re{Lam})),
\be\la{gex}
\psi_0(\ov x)=\sum_{\ov k} C(\ov k)\Lam_{j=1}^{\ov N} e^{ik_j x_j}.
\qquad k_j\in\Xi:=\fr{2\pi}N\Z^3.
\ee
Here  $\ov k:=\{k_1,...,k_{\ov N}\}$, where $k_j$ are different for 
distinct $j$, and $\fr12 \sum_{j=1}^{\ov N}k_j^2=\om_0$.
We will consider the eigenfunctions 
(\re{gex}) with the additional restriction
\be\la{adr}
\#(\ov k\setminus\ov k')\ge 2\quad{\rm if}
\quad \ov k\ne \ov k'. 
\ee
This condition implies that
the corresponding electronic charge density 
is uniform
(see Lemma \re{lre}),
\begin{equation}\label{roZ}
\rho^e(x)\equiv -eZ,\qquad x\in \T.
\end{equation}
This identity plays a crucial role in our approach. It implies that
the corresponding total charge density (\re{Hfor2}) identically vanishes by (\re{sipi}).
Let us emphasize
that both ionic and electronic charge densities are uniform for the ground state under 
the Jellium
condition together with (\re{adr}).
\medskip

\setcounter{equation}{0}
\section{Fermionic jellium ground states}
Here we check the key identity (\re{roZ})
and
construct all solutions to (\re{HS}) with minimal 
energy  (\re{Hfor}). Furthermore we
give examples illustrating the Jellium and 
the Wiener conditions.

\subsection{Uniform electronic charge density}
Let us  establish the   identity (\re{roZ}).
\bl\la{lre} 
Let  the condition (\re{adr}) hold for an
 eigenfunction (\re{gex}), and 
\be\la{norm}
\int_{\ov\T}|\psi_0(\ov x)|^2d\ov x=Z.
\ee
Then the identity  (\re{roZ}) holds. 
\el
\begin{proof}
By the  antisymmetry of  $\psi_0(x_1,...,x_{\ov N})$
it remains to prove that 
\be\la{norm2}
\int_{\ov\T}\de(x-x_1)|\psi_0(\ov x)|^2d\ov x=Z/\ov N,\qquad x\in\T.
\ee
Let us use the expansion (\re{gex}).
The normalization condition (\re{norm}) gives
\be\la{ppo12}
\sum_{\ov k}|C(\ov k)|^2 \ov N^{\ov N}=Z.
\ee
Further,
\beqn
&&
\int_{\ov\T}\de(x-x_1)|\psi_0(\ov x)|^2d\ov x
= \fr 1{\ov N!}\sum_{\ov k}\Big\{|C(\ov k)|^2
\int_{\ov\T}\de(x\!-\!x_1)
\Big[\sum_{\pi,\pi'\in S_{\ov N}} 
(-1)^{|\pi|+|\pi'|} \prod\limits_{j=1}^{\ov N}
e^{i[k_{\pi(j)}-k_{\pi'(j)}]x_j}
\Big]
d\ov x\Big\} \nonumber 
\\
&+&\fr 1{\ov N!}\rRe\sum_{\ov k\ne \ov k'}\Big\{C(\ov k)\ov C(\ov k')
\Big[\sum_{\pi,\pi'\in S_{\ov N}} 
(-1)^{|\pi|+|\pi'|} \int_{\ov\T}\de(x\!-\!x_1)\prod\limits_{j=1}^{\ov N}
e^{i[k_{\pi(j)}-k'_{\pi'(j)}]x_j}
d\ov x\Big]\Big\}. \la{a}
\eeqn
The integrals in the last line vanish 
since $k_{\pi(j)}-k'_{\pi'(j)}\ne 0$
at least for one $j\ne 1$ by (\re{adr}). 
On the other hand,
the integrals in the first line do not vanish
only in the case when $k_{\pi(j)}\equiv k_{\pi'(j)}$ for $j\ne 1$, i.e.,
when $\pi=\pi'$. Hence, 
\be
\int_{\ov\T}\de(x-x_1)|\psi_0(\ov x)|^2d\ov x
= 
 \ov N^{\ov N-1}
\sum_{\ov k}|C(\ov k)|^2
\int_{\T}\de(x\!-\!x_1)
d x_1=\ov N^{\ov N-1}
\sum_{\ov k}|C(\ov k)|^2=Z/\ov N
\ee
by (\re{ppo12}).

\end{proof}
\br\la{rPA}
Similar calculations 
show that the uniformity (\re{norm2}) can break down
for the wave functions (\re{gex}) if the condition (\re{adr}) fails.
\er

\subsection{Description of  ground states}

The following lemma describe all ground states
 with $\Ga$-periodic arrangement of ions.  

\bl\la{Jgs}
All solutions 
to (\re{HS})
of minimal  energy 
with $\Ga$-periodic arrangement of ions  
are given by (\re{gr}), where $\psi_0$ is an eigenfunction (\re{eig}) 
with the normalization (\re{rQ}).

\el
\begin{proof} 
It suffices to construct all solutions $(\psi(t),q(t),p(t))$ which minimize
the first integral  
on the right hand side
of (\re{Hfor}) under the normalization
condition (\re{rQ}),
  with zero  second and the 
third terms and with $\Ga$-periodic arrangement of ions.

First, the  solutions (\re{gr})
have all these properties when $\psi_0$ is the eigenfunction
 (\re{gex}) satisfying the condition  (\re{adr}).
Namely,  the
first integral  
on the right hand side
of (\re{Hfor}) takes the minimal value 
for the eigenfunctions
under the normalization
condition (\re{rQ}). The second and the 
third terms on the right hand side vanish since the corresponding total charge density
$\rho(x)\equiv 0$ by (\re{sipi}) and (\re{roZ}).

Similarly, for general solution $(\psi(t),q(t),p(t))$ 
the first integral, under 
the normalization condition (\re{rQ}),
takes the  
minimal value for the eigenfunctions (\re{eig}).
Then
\be\la{eig2}
-\fr12 \De^\otimes\5\psi(\ov x,t)=
\om_0\psi(\ov x,t),\qquad \ov x\in\ov\T,\quad t\in\R.
\ee
The second summand of (\re{Hfor}) vanishes only for $\rho(x)\equiv 0$.
Then, up to a gauge transformation,
$\Phi(\cdot,t)=G\rho(\cdot,t)= 0$.
Now the equation
  (\re{LPS1}) 
implies that $\psi(\ov x,t)=e^{i\om_0t}\psi_0(\ov x)$ by (\re{eig2}).
Finally, the third summand of (\re{Hfor}) vanishes only for 
 $\dot q(n,t)=p(n,t)\equiv 0$.
Hence, by the $\Ga$-periodicity,
\begin{equation}\la{qnc}
q(n,t)\equiv r,\qquad n\in\Ga,\quad t\in\R,
\end{equation}
where $r\in \T$.
 \end{proof}

 \subsection{The Jellium and Wiener conditions. Examples}\la{sex}

The Wiener condition (\re{W1}) for the ground states (\re{gr}) holds 
under the generic assumption
(\re{W1s}).
On the other hand,  (\re{W1}) does not hold
  for the simplest Jellium model,
 when $\si(x)$ is the  function
\begin{equation}\la{sic}
\si_1(x):=
eZ\chi_1(x)\chi_1(x)\chi_1(x),\qquad x\in\T,
\end{equation}
where $\chi_1$ is the characteristic function of the interval $[0,1]$ mod $N$.
In this case
the Fourier transform 
\begin{equation}\la{sJM}
\ti\sigma_1(\xi)=eZ\ti\chi_1(\xi_1)\ti\chi_1(\xi_2)\ti\chi_1(\xi_3),
\qquad \xi\in\Xi,
\end{equation}
where
\begin{equation}\la{sJM1}
\qquad
\ti\chi_1(s)=\fr {2\sin s/2}s,\quad s\in \fr{2\pi}N\Z\setminus 0.
\end{equation}
Now
for $\theta= (0,\theta_2,\theta_3)$ we have
\begin{equation}\la{DK2}
\Si(\theta)= \sum_{m\in\Z^3:\,m_1=0}\Big[
 \fr{\xi\otimes\xi}{|\xi|^2}|\ti\si(\xi)|^2\Big]_{\xi=\theta+2\pi m},
\qquad \theta\in \Pi^*\setminus\ga^*,
\end{equation}
which is a degenerate matrix since $\xi_1=0$ in each summand. Hence, (\re{W1}) fails.
Similarly, the Wiener condition  fails for 
$\si_k(x)=eZ\chi_k(x_1)\chi_k(x_2)\chi_k(x_3)$,
 where $\chi_k=\chi_1*...*\chi_1$ ($k$ times)
 with $k=2,3,...$,
since in this case
\begin{equation}\la{sJM2}
\ti\sigma_k(\xi)=eZ\ti\chi_k(\xi_1)\ti\chi_k(\xi_2)\ti\chi_k(\xi_3);\qquad
\ti\chi_k(s)=\Big[\fr {2\sin s/2}s\Big]^k,\quad s\in
\fr{2\pi}N\Z\setminus 0.
\end{equation}


\setcounter{equation}{0}
\section{The orbital stability of the ground state}

In this section we expand the energy into the Taylor series and prove the orbital stability checking the
positivity of the energy Hessian.



\subsection{The Taylor expansion of energy functional}

We will deduce the lower estimate (\re{BLi}) using
 the  Taylor expansion of $E(S+Y)$ for 
$S=S_{\al,r}=(\psi_\al,\ov r,0)\in \cS$ and $Y=(\vp,\vka,p)\in\cW=H^1(\ov\T)\oplus \R^{3\ov N}\oplus \R^{3\ov N}$:
 \be\la{te} 
E(S+Y)=E(S)+\langle E'(S),Y\rangle +\fr12 \langle Y,E''(S)Y\rangle + R(S,Y)=\om_0Z+
\fr12 \langle Y,E''(S)Y\rangle + R(S,Y)
\ee
since $E(S)=\om_0Z$ by (\re{ES}), and $E'(S)=0$.
Here  $E'(S)$ and  $E''(S)$ stand for the G\^ ateaux differentials.
Let us recall that $\psi_\al=e^{i\al}\psi_0(x)$ where $\psi_0(x)$
is given by (\re{gex}) and the condition (\re{adr}) holds.

First,
we expand 
 the  charge density (\re{Hfor2})
corresponding to $S+Y=(\psi_\al+\vp,\ov r+\vka,p)$:
\be\la{ro}
\rho(x)=\rho^{(0)}(x)+\rho^{(1)}(x)+\rho^{(2)}(x),\qquad x\in\T,
\ee
where $\rho^{(0)}$ and  $\rho^{(1)}$ are respectively 
the  terms of zero and first  order in 
$Y$, while  $\rho^{(2)}$ is the remainder.
However, $\rho^{(0)}(x)$ is the total charge density of the ground state
which is identically zero
by
(\ref{sipi}) and (\re{roZ}):
\begin{equation}\la{ro0}
\rho^{(0)}(x)=\rho^i_0(x)-e|\psi_\al(x)|^2\equiv 0,\qquad x\in{\T}.
\end{equation}
Thus, $\rho=\rho^{(1)}+\rho^{(2)}$.
Expanding (\re{Hfor2}) further, we obtain
\beqn
\la{ro11}
\!\!\!\!\!\!\!\!\!\!\!\!\!\!\!\!\!\!\!\!\!\!\!\!\!\!\rho^{(1)}(x)\!\!\!\!&\!\!\!\!=\!\!\!\!&\!\!
\si^{(1)}(x)
-2e\,\sum\limits_{j=1}^{\ov N}\rRe(\psi_\al, \vp)_j(x)
,\quad \si^{(1)}(x)=-\sum_{n\in\Ga} \vka(n)\cdot\na\si(x-n-r),\quad
\\
\la{ro13}
\!\!\!\!\!\!\!\!\!\!\!\!\!\!\!\!\!\!\!\!\!\!\!\!\!\!\rho^{(2)}(x)\!\!\!\!&\!\!=\!\!\!\!&\!\!\si^{(2)}(x)
\!-\!e
\sum\limits_{j=1}^{\ov N}(\vp,\vp)_j(x),~ \si^{(2)}(x)\!=\!\fr12\sum_{n\in\Ga}
\int_0^1\!\!(1\!-\!s)
[\vka(n)\cdot\na]^2\si(x\!-\!n\!-\!r-\!s\vka(n))ds,
\eeqn
where we denote
\be\la{wwd}
(\psi_\al, \vp)_j(x):=\int_{\ov\T}\de(x-x_j)\psi_\al(\ov x)\ov\vp(\ov x)\,d\ov x,
\qquad x\in\T.
\ee
Substituting  $\psi=\psi_\al +\vp$ and $\rho=\rho^{(1)}+\rho^{(2)}$ 
 into (\ref{Hfor}), we obtain that
the  quadratic part  of (\re{te}) reads 
\begin{equation}\la{B2}
\fr12\langle Y,E''(S) Y \rangle=\fr12\int_{\ov\T}|\na \vp(\ov x)|^2]d\ov x+
\fr12 (\rho^{(1)},G\rho^{(1)})+K(p),~~~~~~ K(p):=\ds\sum_n\fr{p^2(n)}{2M}~~
\end{equation}
and the remainder equals
\begin{equation}\la{B3}
R(S,Y)=\fr 12 (2\rho^{(1)}+\rho^{(2)},G\rho^{(2)}). 
\end{equation}

\subsection{The null space of Hessian}
In this section
 we calculate the null space 
\be\la{KYd}
\cK(S):=\Ker\, \Big[E''(S)\Big|_\cW\Big],\qquad S\in \cS.
\ee

\bl\la{lW}
Let the Jellium condition (\re{Wai}) and the Wiener condition   
(\re{W1}) hold, and $S\in\cS$. Then
\be\la{KY}
\cK(S)=\{(0,\ov s,0):~~s\in\R^3 \},
\ee
where $\ov s\in\R^{3\ov N}$ is defined similarly to (\re{gr2}): 
$\ov s(n)\equiv s$.
\el
\begin{proof}
	All the summands of
	the
	energy (\re{B2}) are nonnegative. Hence,  this expression is zero if and only if
	all the summands vanish: in the notation (\re{ro11})
	\be\la{rb}
	\vp(\ov x)\equiv C,\quad
	(\rho^{(1)}, G\rho^{(1)})=\Vert\sqrt{G}[\si^{(1)}-2e
	\sum\limits_{j=1}^{\ov N}\rRe(\psi, \vp)_j(x)
	]\Vert_{L^2(\T)}^2
	=0,\quad p=0.
	\ee
	Here $C=0$ by the antisymmetry of $\vp$. Therefore, 
	 $(\psi, \vp)_k(x)\equiv 0$, and hence,
 (\re{rb}) implies that
	\be\la{rb2}
	\sqrt{G}\si^{(1)}=0.
	\ee
	On the other hand, 
	in the Fourier transform 
	(\re{ro11}) reads
	\be\la{B314}
	\ti\si^{(1)}(\xi)=\ti\si(\xi)\xi\cdot\sum_{n\in\Ga} ie^{i\xi [n+r]}\vka(n)
	=i\ti\si(\xi)\xi\cdot  e^{i\xi r} \hat \vka(\xi),\qquad\xi\in \Xi,
	\ee
	where $\hat \vka(\xi):=\sum_{n\in\Ga} e^{i\xi n}\vka(n)$ is a $2\pi\Z^3$-periodic function on
	$\Xi$.
	Hence, Definition (\re{fs}) and the Jellium condition (\re{Wai}) imply that
	\beqn\la{B315}
	0=\Vert\sqrt{G}\si^{(1)}\Vert_{L^2(\T)}^2&=&
	N^{-3}\sum_{\Xi\setminus \ga^*} |\ti\si(\xi)\fr{\xi\hat \vka(\xi)}{|\xi|}|^2
	\nonumber\\
	\nonumber\\
	&=&N^{-3}\sum_{\theta\in\Pi^*\setminus \ga^*} 
	\langle\hat \vka(\theta),
	\sum_{m\in\Z^3}\Big[\fr{\xi\otimes\xi}{|\xi|^2}|\ti\si(\xi)|^2\Big]_{\xi=\theta+2\pi m}\hat \vka(\theta)\rangle
	\nonumber\\
	\nonumber\\
	&=&N^{-3}\sum_{\theta\in\Pi^*\setminus \ga^*} 
	\langle\hat \vka(\theta),
	\Si(\theta)
	\hat \vka(\theta)\rangle.
	\eeqn
	As a result, 
	\be\la{B316}
	\hat \vka(\theta)=0,\qquad \theta\in \Pi^*\setminus \ga^*
	\ee
	by the Wiener condition
	(\re{W1}).
	On the other hand, $\hat \vka(0)\in\R^3$ remains arbitrary, see Remark \re{rW} ii).
	Respectively, 
	$\vka=\ov s$ with 
	an arbitrary $s\in\R^3$.\end{proof}
\br
{\rm
The key point of the proof is the explicit calculation (\re{B314}) 
in the Fourier transform.
This calculation relies on the invariance of
the Hessian $E''(S)$ with respect to $\Ga$-translations which is 
due to the periodicity of the ions arrangement of the ground state.
}
\er

\br{\it  Beyond the Wiener condition}
{\rm
If the Wiener
condition (\re{W1}) fails, the dimension of the space
\be\la{V}
V:=\{v\in \R^{3\ov N}:~~ v(n)=\sum_{\theta\in\Pi^*\setminus\ga^*} e^{-i\theta n} \hat v(\theta),
\qquad \hat v(\theta)\in\C^3,~~ \Si(\theta)\hat v(\theta)=0\}
\ee
 is positive.
The above calculations  show that in this  case
\be\la{KYg}
\cK(S)=\{(0, \ov s+v,0):~~s\in\R^3,~~v\in V \}.
\ee
The subspace $V\subset \R^{3\ov N}$ is orthogonal to the $3D$ subspace 
$\{\ov s:s\in\R^3\}\subset \R^{3\ov N}$
by the Parseval theorem.
Hence,
$\dim \cK(S)=3+d$, 
where $d:=\dim V>0$. 
Thus, $\dim \cK(S)>3$.
Under the Wiener condition $V=0$, and (\re{KYg}) coincides with (\re{KY}).
}
\er

\subsection{The positivity  of  Hessian}
Denote by
$ N_S\cS$ the normal subspace to $\cS$ at a point $S$:
\be\la{L0N}
 N_S\cS:=\{Y\in\cW=H^1(\ov\T)\oplus \R^{3\ov N}\oplus \R^{3\ov N}: \langle Y,\tau\rangle=0,~~\tau\in  T_S\cS\},
\ee
where $ T_S\cS$ is the
tangent space to $\cS$ at the point $S$ and $\langle\cdot,\cdot\rangle$
stands for the scalar product (\re{dWW}).
Obviously, $\cS\subset\cM$ and  the tangent space to $\cM$ at a point
$S=(\psi_\al,\ov r, 0)$
is given by 
 \begin{equation}\la{TSM}
T_S\cM=\{(\vp, \vka,\pi)\in\cW:\vp\bot\psi_\al,~~ \vka\in\R^{3\ov N},
~~\pi\in\R^{3\ov N} \},
\end{equation}
since $DQ(\psi_\al,\ov r,0)=2(\psi_\al,0,0)$.

\bl\la{lW2}
Let
 the 
Jellium condition (\re{Wai}) hold, and $S=S_{\al,r}\in \cS$. Then the Wiener condition (\re{W1}) is necessary and sufficient
for the positivity of the Hessian 
$E''(S)$
in the orthogonal directions to $\cS$ on $\cM$,
i.e., 
\be\la{cM2}
E''(S)\Big|_{N_S\cS\cap  T_S\cM}>0.
\ee
\el
\begin{proof} i) Sufficiency.
Differentiating $S_{\al,r}=(e^{i\al}\psi_0,\ov r,0)\in \cS$ in the parameters $\al \in [0,2\pi]$ and $r\in\T$, 
we obtain 
\begin{equation}\la{tv}
T_S\cS=\{(iC\psi_\al ,\ov s,0): ~~C\in\R,~~s\in\R^3\}.
\end{equation}
	Hence,
	(\re{KY}) implies that 
	\be\la{tv2}
	\cK(S)\cap  N_S\cS=
	(0 ,0,0)
	\ee
Now  (\re{cM2}) follows since $E''(S)\ge 0$ by (\re{B2}).
	\medskip\\
 ii)	Necessity.
If the Wiener condition (\re{W1}) fails,
the space $\cK(S)$ is given by (\re{KYg}), and hence, 
(\re{tv}) implies that now
\be\la{tv3}
\cK(S)\cap  N_S\cS=
\{0 ,v,0): ~~C\in\R, ~~ v\in V\}\subset T_S\cM.
\ee
Therefore, the Hessian $E''(S)$ vanishes on the nontrivial space
$\cK(S)\cap  N_S\cS\subset   T_S\cM$ of the
dimension  $d>0$. 
Respectively, the positivity (\re{cM2}) breaks down.
\end{proof}

\br\la{rS} {\rm The positivity of type (\re{cM2}) breaks down for the submanifold 
	$\cS(r):=\{S_{\al ,r}:~\al\in[0,2\pi]\}$ 
	with a fixed $r\in\T$ instead of  the solitary manifold $\cS$.
	Indeed, then the corresponding tangent space  is smaller:
	\be\la{tvr}
	 T_S\cS(r)=\{(iC\psi_\al ,0,0):\,\,C\in\R\}.
	\ee
	Hence, the normal subspace $ N_S\cS(r)$ is larger, in particular containing all 
the vectors $(0,\ov s,0)$
	generating  the shifts of the torus. However, all these vectors also 
belong 
	to the null space (\re{KY}) and to $ T_S\cM$.
	Respectively, the null space of the Hessian $E''(S)$ in $ T_S\cM\cap N_S\cS(r)$ is  at least 3-dimensional.
}
\er

\subsection{The orbital stability}
Here we prove  Theorem \re{tm} which is our  main result.
For the proof is suffices to check
the lower energy estimate (\re{BLi}): 
\be\la{BL}
E(X)-\om_0Z\ge \nu\,d^2_\cV(X,\cS)\quad{\rm if}\quad d_\cV(X,\cS)\le \de,
\quad X\in\cM
\ee
with some $\nu,\de>0$. 
This estimate implies Theorem \re{tm} since the energy is conserved
along all trajectories.
First, we prove similar lower bound for the energy Hessian. 
\bl
Let conditions of Theorem \re{tm} hold. Then  for each $S\in \cS$
\be\la{L02}
\langle Y, E''(S)Y\rangle > \nu\Vert Y\Vert_\cW^2,  \qquad Y\in N_S\cS\cap   T_S\cM ,
\ee
 where $\nu>0$.
\el
\begin{proof}
It suffices to prove this estimate for $S=(\psi_0,0,0)$.
First, we note that $E''(S)$ is not complex linear due to the integral 
in (\re{Hfor}). Hence, we should express the action of $E''(S)$ 
in $\psi_1(x):=\rRe\psi(x)$ and $\psi_1(x):=\rIm\psi(x)$:
by the formula (1.15) of \ci{KKpl2015},
\begin{equation}\la{E''}
E''(S)Y=
\left(\begin{array}{cccl}
  -\De^\otimes+4e^2\psi_0 G\psi_0 & 0 & 2L & 0
 \medskip\\
 0 &   -\De^\otimes  &0 & 0\medskip\\
 2L^{\5*}  &    0  &   T    & 0  \\
      0      &    0            &   0    &  M^{-1} \\
\end{array}\right)Y\qquad{\rm for}\quad
Y=\left(\begin{array}{c}\psi_1 \\ \psi_2 \\ q \\ p \end{array}\right),
\end{equation}
where 
$\psi_0$
denotes the operators of  multiplication
by the real function $\psi_0(x)\equiv\sqrt{Z}$.
The operators $L$  correspond to
the matrix
\begin{equation}\la{S}
 L(x,n):=e\psi_0(x)G\na\si(x-n):~~~x\in\R^3,~n\in\Ga
\end{equation}
by formula (3.3) of \ci{KKpl2015}
and  $T$ corresponds to the real matrix with entries
\begin{equation}\la{T}
T(n-n'):=-\ds\langle  G\na\otimes\na\si(x-n'),  \si(x-n) \rangle,\qquad n,n'\in\Ga
\end{equation}
by formula (3.4) of \ci{KKpl2015} since the corresponding potential 
$\Phi_0=0$.
Hence,
$E''(S)$ is a finite-rank perturbation of the operator 
with the discrete spectrum
on the torus ${\T}$. 
Finally,  (\re{cM2}) implies that the minimal eigenvalue of
$E''(S)$ is  positive.
Therefore, (\re{L02}) follows.
\end{proof}
\medskip

The positivity (\re{L02}) implies
the lower energy estimate (\re{BL})  since the 
higher-order terms in (\re{te}) 
are negligible by the following  lemma.

\bl\la{lre2}
Let $\si(x)$ satisfy  (\re{ro+}).
Then the  remainder  (\re{B3}) admits the bound
\be\la{B31}
|R(S,Y)|\le C
\Vert Y\Vert_\cW^3
\quad\,\,\,{\rm for}\quad  \,\,\,
\Vert Y\Vert_\cW \le 1.
\ee
\el
\begin{proof} Due to (\re{B3}) 
	it suffices to prove the 
	estimates
	\be\la{B312}
	\Vert\sqrt{G}\rho^{(1)}\Vert_{L^2(\T)}\le C_1\Vert Y\Vert_\cW,\quad  
	\Vert\sqrt{G}\rho^{(2)}\Vert_{L^2(\T)}\le C_2\Vert Y\Vert_\cW^2
	\quad \,\,\,{\rm for}\,\,\,\quad     \Vert Y\Vert_\cW    \le 1.
	\ee
	i)
	By (\re{ro11}) we have
	for $Y=(\vp,\vka,p)$ 
	\be\la{B313}
	\sqrt{G}\rho^{(1)}=\sqrt{G}\si^{(1)}-2e\sqrt{G}
	\sum\limits_{j=1}^{\ov N}\rRe(\psi, \vp)_j(x).
	\ee
	in the notation (\re{wwd}).
	The operator $\sqrt{G}$ is bounded in $L^2(\R^3)$ by (\re{fs}). 
	Hence,  (\re{ro11}) implies that
	\be\la{GP1}
	\Vert\sqrt{G}\si^{(1)}\Vert_{L^2(\T)}\le C |\vka|.
	\ee
 Applying the Cauchy--Schwarz and Hausdorff--Young inequalities
to the second term on the RHS of  (\re{B313}),
	we obtain 
	\beqn\la{GP2}
	\Vert\sqrt{G}
	(\psi, \vp)_j
	\Vert_{L^2(\T)}&\le& 
	C\Big[\sum_{\xi\in\Xi\setminus 0}\fr{|\ti\vp(\xi)|^2}{|\xi|^2}\Big]^{1/2}
\le C\Vert\ti\vp\Vert_{L^4(\Xi)}\Big[\sum_{\xi\in\Xi\setminus 0}|\xi|^{-4}\Big]^{1/2}
\nonumber\\
&\le& C_1\Vert\vp\Vert_{L^{4/3}(\T)}\le C_2\Vert\vp\Vert_{H^1(\T)}^2
	\eeqn
by the Sobolev embedding theorem.
	Hence, the first inequality (\re{B312}) is proved.
	\medskip\\
	ii) Now we prove the second 
	inequality  (\re{B312}).  According to (\re{ro13}),
	\be\la{B317}
	\sqrt{G}\rho^{(2)}(x)=\sqrt{G}\si^{(2)}(x)-e\sqrt{G}
	\sum\limits_{j=1}^{\ov N}(\vp,\vp)_j(x).
	\ee
	Similarly to  (\re{GP1})
	\be\la{GP3}
	\Vert\sqrt{G}\si^{(2)}\Vert_{L^2(\T)}\le C|\vka|^2.
	\ee
	At last, 
	denoting $\beta(x):=(\vp,\vp)_j(x)$, we obtain similarly to (\re{GP2}) 
	\beqn\la{fp}
	\Vert\sqrt{G}
	(\vp,\vp)_k
	\Vert_{L^2(\T)}
	\le
	C\Big[\sum_{\xi\in\Xi\setminus 0}\fr{|\ti \beta(\xi)|^2}{|\xi|^2}\Big]^{1/2}
	\le
	C_1\Vert \beta\Vert_{L^{4/3}(\T)}.
	\eeqn
Finally,
applying the triangle inequality and the Sobolev embedding theorem, we obtain
\beqn\la{fp2}
\Vert \beta\Vert_{L^{4/3}(\T)}
&\le&
\int_{\T^{\ov N-1}} [\int_\T|\vp(\ov x)|^{8/3} dx_j]^{3/4}\,dx_1...
\widehat{dx_j}
...dx_{\ov N}
\nonumber\\
&\le&
\int_{\T^{\ov N-1}} [\int_\T|\na_{x_j}\vp(\ov x)|^2dx_j]\,dx_1...
\widehat{dx_j}
...dx_{\ov N}
\le
 C\Vert\vp\Vert_{H^1(\T)}^2.
\eeqn
Now the lemma is proved.
\end{proof}

\appendix

\setcounter{section}{0}
\setcounter{equation}{0}
\protect\renewcommand{\thesection}{\Alph{section}}
\protect\renewcommand{\theequation}{\thesection.\arabic{equation}}
\protect\renewcommand{\thesubsection}{\thesection.\arabic{subsection}}
\protect\renewcommand{\thetheorem}{\Alph{section}.\arabic{theorem}}

\setcounter{equation}{0}
\section{Global dynamics}

Here we prove the global well-posedness of the 
system (\ref{HS}).

\bt\label{TLWP1}
Let   (\re{ro+}) hold and $X(0)\in\cM$.
Then 
\medskip\\
i) 
There exists a unique
solution 
$X(t)\in C(\R,\cV)$ 
to (\re{HS}).
\medskip\\
ii) The energy and charge conservations (\re{EQ}) hold.

\et

First we construct the local solutions by contraction arguments.  
To construct the global solutions we prove
in Appendix B  energy conservation 
using the Galerkin approximations.

Let us prove the local well-posedness.

\bt\label{TLWP}(Local well-posedness).
Let   (\re{ro+}) hold and
$X(0)=(\psi_0,q_0,p_0)\in\cV=H^1(\ov\T)\otimes \ov\T\otimes \R^{3\ov N}$
with
 $| X(0)|_\cV:=\Vert\psi_0\Vert_{H^1(\ov\T)}+|p_0|\le R$. Then 
there exists $\tau=\tau(R)>0$ such that
  equation (\ref {HS}) has  a unique  solution $X\in C([-\tau,\tau],{\cal V})$,
 and the maps $U(t):X(0)\mapsto X(t)$ are continuous in $\cV$ for $t\in [-\tau,\tau]$.

\et
In the
 next two propositions we prove
the boundedness and the local Lipschitz continuity of the nonlinearity $N:\cV\to\cW=H^1(\ov\T)\oplus \R^{3\ov N}\oplus \R^{3\ov N}$ defined in (\re{HN}).
With  this proviso Theorem \re{TLWP} follows from the integral form 
(\re{LPSiv}) of the equation (\ref{HS})
 by the contraction mapping principle, since $e^{-At}$ is an isometry of $\cW$.
First, we prove the boundedness of $N$.
\bp\label{p1}
For any $R>0$ and $X=(\psi,q,p)\in\cV$
\begin{equation}\label{bN}
\Vert N(X)\Vert_\cW\le C(R)\qquad{\rm for}\quad |X|_\cV\le R.
\end{equation}
\ep
\begin{proof}
We need appropriate bounds for the charge density   $\rho$
and for the corresponding potential  $\Phi$.

\bl The charge density (\re{Hfor2}) admits the bounds 
\be\la{re4}
\Vert\rho\Vert_{L^3(\T)}+\Vert\na\rho\Vert_{L^{3/2}(\T)}\le C(1+\Vert\psi\Vert_\cF^2).
\ee

\el
\begin{proof}
 We split $\rho(x)$ as
  $\rho(x)=\rho^i(x)+\rho^e(x)$,
where
\[
\rho^i(x,t)=\sum_{n\in\Ga}\sigma(x-n-q(n,t)),
\]
while $\rho^e$ is defined by (\re{re}).
The bound
(\re{re4})
for $\rho^i$ holds by (\re{ro+}). It remains to prove the bound for $\rho^e$.
Definition (\re{re})   implies that  
\be\la{re2}
\rho^e(x)=-e\sum_{j=1}^{\ov N} \int_{\T^{\ov N-1}} |\psi(\ov x)|^2\Big|_{x_j=x}\,dx_1...\widehat{dx_j}...dx_{\ov N},
\quad x\in\T,
\ee
where the hat means that this differential is omitted. 
Differentiating, we obtain that
\be\la{re22}
\na\rho^e(x)=-e\sum_{j=1}^{\ov N} \int_{\T^{\ov N-1}} \na_{x_j}|\psi(\ov x)|^2\Big|_{x_j=x}\,dx_1...\widehat{dx_j}
...dx_{\ov N},\quad x\in\T.
\ee
Applying the triangle inequality to (\re{re2}), we get 
\beqn\la{re23}
\Vert\rho^e(x)\Vert_{L^3(\T)}&\le& C\sum_{j=1}^{\ov N} \int_{\T^{\ov N-1}} [\int_\T|\psi(\ov x)|^6dx_j]^{1/3}\,dx_1...
\widehat{dx_j}
...dx_{\ov N}
\nonumber\\
&\le&
\int_{\T^{\ov N-1}} [\int_\T|\na_{x_j}\psi(\ov x)|^2dx_j]\,dx_1...
\widehat{dx_j}
...dx_{\ov N}
\le
 C\Vert\psi\Vert_{H^1(\ov\T)}^2.
\eeqn
by the Sobolev embedding theorem \ci[Theorem 5.4, Part I]{Adams}.
Similarly, (\re{re22}) implies that
\beqn\la{re24}
\Vert\na\rho^e(x)\Vert_{L^{3/2}(\T)}&\le& C\sum_{j=1}^{\ov N}
\int_{\T^{\ov N-1}} [\int_\T|\psi(\ov x)\na_{x_j}\ov\psi(\ov x)|^{3/2}dx_j]^{2/3}\,dx_1...
\widehat{dx_j}
...dx_{\ov N}
 \nonumber\\
&\le&
\int_{\T^{\ov N-1}} [\int_\T|\psi(\ov x)|^6dx_j]^{1/6}[\int_\T|\na_{x_j}\psi(\ov x)|^2dx_j]^{1/2}\,dx_1...
\widehat{dx_j}
...dx_{\ov N}
 \nonumber\\
&\le&
\int_{\T^{\ov N-1}} [\int_\T|\na_{x_j}\psi(\ov x)|^2dx_j]\,dx_1...
\widehat{dx_j}
...dx_{\ov N}
\le
 C\Vert\psi\Vert_{H^1(\ov\T)}^2
\eeqn
by the H\"older inequality and  the Sobolev embedding theorem.
\end{proof}

\bl
 The potential $\Phi:=G\rho$ admits the bound
\be\la{Pb}
\Vert \Phi\Vert_{C(\T)}+\Vert \na\Phi\Vert_{L^3(\T)}\le C(1+\Vert\psi\Vert_\cF^2).
\ee
\el
\begin{proof}
 Applying   the H\"older and Hausdorff--Young inequalities to 
(\re{fs}), we obtain that
\be\la{s1}
\Vert \Phi\Vert_{C(\T)}\le C\Vert\fr{\ti\rho(\xi)}{\xi^2}\Vert_{L^1(\Xi\setminus 0)}
\le C_1\Vert \xi\tilde \rho\Vert_{L^3(\Xi)}
\Big[\sum_{\xi\in\Xi\setminus 0}|\xi|^{-9/2}\Big]^{2/3}
\le C_2\Vert\nabla \rho\Vert_{L^{3/2}(\T)}.
\ee
Similarly, 
\be\la{s2}
\Vert \na\Phi\Vert_{L^3(\T)}\le C\Vert\fr{\ti\rho(\xi)}{|\xi|}\Vert_{L^{3/2}(\Xi\setminus 0)}
\le C_1\Vert \xi\tilde \rho\Vert_{L^3(\Xi)}
\Big[\sum_{\xi\in\Xi\setminus 0}|\xi|^{-6}\Big]^{1/3}
\le C_2\Vert\nabla \rho\Vert_{L^{3/2}(\T)}.
\ee
Now the bound (\re{Pb}) follows from (\re{re4}).
\end{proof}

Now we can prove the estimate (\re{bN}).
First, we will prove
\begin{equation}\label{eq2}
\Vert\Phi^\otimes\5\psi\Vert_\cF\le C(1+\Vert\psi\Vert_\cF^3)
\end{equation}
in the notation (\re{2De}).
According to definition (\re{FN})
it suffices to  check that
\begin{equation}\la{c1}
\Vert\Phi^\otimes\5\psi \Vert_{L^2(\ov\T)}+
\Vert \Phi^\otimes\, 
\nabla^\otimes\5\psi\Vert_{L^2(\ov\T)}
+
\Vert \psi\nabla^\otimes\,\Phi^\otimes\Vert_{L^2(\ov\T)}\le C(1+\Vert\psi\Vert_\cF^3).
\end{equation}
The first two summands   admit the needed estimate by (\re{Pb}). The third summand
requires some additional argument. Namely,
\beqn\la{aa}
\Vert \psi 
\nabla^\otimes\,\Phi^\otimes\5\Vert_{L^2(\ov\T)}^2
&=& \int_{\ov\T}|\sum_{j=1}^{\ov N} \na\Phi(x_j)\psi(\ov x)|^2\,d\ov x
\le C\sum_1^{\ov N} \int_{\ov\T}|\na\Phi(x_j)\psi(\ov x)|^2\,d\ov x
\nonumber\\
 &=&C\sum_{j=1}^{\ov N} \int_{\T^{\ov N-1}}
 \Big[\int_{\T}|\na\Phi(x_j)\psi(\ov x)|^2\,d x_j\Big]dx_1...\widehat{dx_j}...dx_{\ov N}.
\eeqn
The inner integral is estimated as follows
\be\la{ii}
\int_\T|\na\Phi(x_j)\psi(\ov x)|^2\,d x_j\le \Vert \na\Phi\Vert_{L^3(\T)}^2\Big[\int _\T|\psi(\ov x)|^6\,d x_j\Big]^{1/3}
\le C\Vert \na\Phi\Vert_{L^3(\T)}^2\int_\T|\na_{x_j}\psi(\ov x)|^2\,d x_j
\ee
by the H\"older inequality and the Sobolev embedding theorem. Substituting this estimate into (\re{aa}), we obtain
\begin{equation}\label{c3}
\Vert \psi 
\nabla^\otimes\,\Phi^\otimes\Vert_{L^2(\ov\T)}\le C\Vert \na\Phi\Vert_{L^3(\T)}
\Vert\psi\Vert_\cF.
\end{equation}
This and (\re{Pb}) imply (\re{c1}) for the third summand. 
Finally, using 
(\re{HN}),
(\ref{Pb}) and (\ref{ro+}),
\begin{equation}\label{eq3}
|f(n)|\le \Vert \Phi\Vert_{C(\T)}\Vert\na\si\Vert_{L^1(\T)}
\le C(1+\Vert\psi\Vert_\cF^2),\qquad n\in\Ga.
\end{equation}
Hence, (\ref{eq2}) and (\ref{eq3}) imply (\ref{bN}). Proposition \re{p1} is proved.
\end{proof}

It remains to prove  that the nonlinearity is locally Lipschitz.

\begin{proposition}\label{p2}

For any $R>0$ and $X_1,X_2\in\cV=H^1(\ov\T)\otimes \ov\T\otimes \R^{3\ov N}$ with $| X_1|_{\cal V},| X_2|_\cV\le R$
\begin{equation}\label{lN}
\Vert N(X_1)-N(X_2)\Vert_\cW\le C(R) d_\cV(X_1,X_2).
\end{equation}
\end{proposition}
\begin{proof}
Writing  $X_k=(\psi_k,q_k,p_k)$ and $\Phi_k=G\rho_k$, we obtain that
\begin{equation}\label{NN1}
\Vert\Phi^\otimes_1\5\psi_1-\Phi^\otimes_2\5\psi_2\Vert_\cF\le\Vert(\Phi^\otimes_1-\Phi^\otimes_2)\5\psi_1\Vert_\cF
+\Vert\Phi^\otimes_2 \5(\psi_1-\psi_2)\Vert_\cF.
\end{equation}
Using   (\ref{c3}) and (\ref{Pb}), we obtain 
\begin{eqnarray}\nonumber
\Vert \Phi^\otimes_2\5(\psi_1-\psi_2)\Vert_\cF&\le&
(\Vert \Phi_2\Vert_{C(\T)}+\Vert \nabla \Phi_2\Vert_{L^3(\T)})
\Vert\psi_1-\psi_2\Vert_\cF\\
\nonumber\\
&\le&C(1+R^2)\Vert\psi_1-\psi_2\Vert_\cF\le C(R)
d_\cV(X_1,X_2).\label{NN2}
\end{eqnarray}
Further,
 (\ref{s1}) and (\ref{s2}) give that
\be\la{P12}
\Vert\Phi_1-\Phi_2\Vert_{C(\T)}+\Vert\na\Phi_1-\na\Phi_2\Vert_{L^3(\T)}
\le\Vert\na\rho_1-\na\rho_2\Vert_{L^{3/2}(\T)}
\ee
However, 
 $|\si(x)-\si(x-a)|\le C|a|$,
where $|a|:=\min_{r\in a} |r|$ for $a\in \T$ (by definition, $a\subset\R^3$ is a class of equivalence mod $N\Z^3$).
Therefore,
as in (\ref{re4}),
\be\la{re43}
\Vert\na(\rho_1-\rho_2)\Vert_{L^{3/2}(\T)}\le CR
(| q_1-q_2|+\Vert\psi_1-\psi_2\Vert_\cF).
\ee
Hence,
\beqn
\Vert(\Phi^\otimes_1-\Phi^\otimes_2)\psi_1\Vert_\cF&\le&
(\Vert \Phi_1-\Phi_2\Vert_{C(\T)}+\Vert \nabla (\Phi_1-\Phi_2)\Vert_{L^3(\T)})\Vert\psi_1\Vert_\cF
\nonumber\\
\nonumber\\
&\le&
C R\Vert\na(\rho_1-\rho_2)\Vert_{L^{3/2}(\T)}\Vert\psi_1\Vert_\cF
\le C(R)d_\cV(X_1,X_2).  \label{ek}
\eeqn
Now (\re{NN1})--(\re{ek}) give
\be\la{Th}
\Vert\Phi^\otimes_1\5\psi_1-\Phi^\otimes_2\5\psi_2\Vert_{H^1(\T)}\le C(R)
d_\cV(X_1,X_2).
\ee
Similarly,  (\re{P12}), (\re{re43}) and (\ref{Pb})  imply 
\begin{eqnarray}\nonumber
&&\Vert\langle\na \Phi_1,\sigma(\cdot-n-q_1(n))\rangle-\langle\na\Phi_2,\sigma(\cdot-n-q_2(n))\rangle\Vert
\nonumber\\
\nonumber\\
&\le&\Vert\langle\na(\Phi_1-\Phi_2),\sigma(\cdot-n-q_1(n))\rangle\Vert
+\Vert\langle\nabla \Phi_2,\sigma(\cdot-n-q_1(n))-\sigma(\cdot-n-q_2(n))\rangle\Vert
\nonumber\\
\nonumber\\
&\le&C(\Vert\Phi_1-\Phi_2\Vert_{C(\T)}+\Vert \Phi_2\Vert_{C(\T)}) | q_1-q_2|)\le C(R)
d_\cV(X_1,X_2).
\end{eqnarray}
This estimate together with  (\re{Th}) prove (\re{lN}).
\end{proof}
Now 
Theorem \re{TLWP} follows from 
Propositions \ref{p1} and \ref{p2}. 
\medskip\\
{\bf Proof of Theorem \re{TLWP1}.}
The local solution  $X\in C([-\tau,\tau],\cV)$ to  (\ref {HS})
exists and is unique by Theorem \re{TLWP}.
On the other hand,
the conservation laws (\ref{EQ}) (proved in
Proposition \re{lgal} iii))
together with   (\re{EQV}) imply a priori bound
	\be\la{ab}
	| X(t)|_\cV^2\le C[E(X(0))+ Q(X(0))],\qquad t\in [-\tau,\tau]
	\ee
	by (\re{EQV}).
Hence, the local solution
admits an extension to the global one
$X\in C(\R,\cV)$. 
\hfill$\Box$

\br
The condition $X(0)\in\cM$ implies that $X(t)\in\cM$ for all $t\in\R$ by 
  the charge conservation 
(\re{EQ}). 
	Hence,  (\ref {HS}) implies (\ref{LPS1})--(\ref{LPS3}) with 
	the potential $\Phi(\cdot,t)=G\rho(\cdot,t)$.
\er

\setcounter{equation}{0}
\section{Conservation laws}

We deduce the conservation  laws (\ref{EQ}) by the Galerkin approximations \ci{Lions}. 
Let us recall that
the exterior product of functions 
$f_j\in L^2(\T)$
is defined by
\be\la{Lam}
[\Lam_{j=1}^{\ov N} f_j](\ov x):=\fr1{\sqrt{\ov N!}} \sum_{\pi\in S_{\ov N}} 
(-1)^{|\pi|} \prod\limits_{j=1}^{\ov N}f_j(x_{\pi(j)}),
\qquad \ov x=(x_1,...,x_{\ov N})\in\ov\T,
\ee
where $S_{\ov N}$ is the symmetric group and $|\pi|$ denotes the sign (or parity) of a 
transposition $\pi$.

\bd
i)$\cV_m$ with $m\in\N$  denotes the finite-dimensional  submanifold of 
the Hilbert manifold
$\cV=H^1(\ov\T)\otimes \ov\T\otimes \R^{3\ov N}$ 
\be\la{Vm}
\cV_m:=\{(\sum_{\ov k} C(\ov k)\Lam_{j=1}^{\ov N} e^{ik_jx_j},q,p):k_j\in\Xi,~~
C(\ov k)\in\C,~~
\sum\limits_{j=1}^{\ov N}
k_j^2\le m,~ 
q\in\ov\T,~ p\in\R^{3\ov N}\},
\ee
where $\ov k:=(k_1,...,k_{\ov N})$.
\medskip\\
ii) $\cW_m$ with $m\in\N$  denotes the finite-dimensional linear subspace of the Hilbert space
$\cW=H^1(\ov\T)\oplus \R^{3\ov N}\oplus \R^{3\ov N}$
\be\la{Wm}
\cW_m:=\{(\sum_{\ov k} C(\ov k) \Lam_{j=1}^{\ov N} e^{ik_jx_j},\vka,v):k_j\in\Xi,~~
C(\ov k)\in\C,~~
\sum\limits_{j=1}^{\ov N}k_j^2\le m,  ~ \vka\in\R^{3\ov N}, ~ v\in\R^{3\ov N}\}.
\ee
\ed
Obviously, $\cV_1\subset\cV_2\subset...$, 
the union $\cup_m\cV_m$ is dense in $\cV$, and 
$\cW_m$ are  invariant with respect 
to $H$ and $J$.
Let us denote by $P_m$
the orthogonal projector $\cX\to\cW_m$. This projector is also orthogonal in 
the Hilbert space $\cW$.
Let us 
approximate the system 
(\re{HS}) by finite-dimensional Hamilton systems  on the manifold $\cV_m$,
\be\la{gal}
\dot X_m(t)=JE_m'(X_m(t)),\qquad t\in\R,
\ee
where $E_m:=E|_{\cV_m}$ and $X_m(t)=(\psi_m(t),q_m(t),p_m(t))\in C(\R,\cV_m)$.
The equation (\re{gal}) can be also  written as
\begin{equation}\la{gali}
\langle \dot X_m(t),Y\rangle=-\langle E'(X_m(t)),JY\rangle,\qquad Y\in\cW_m.
\end{equation}
This form of the equation (\re{gal}) holds 
since $E_m:=E|_{\cV_m}$ and $\cW_m$ is  invariant with respect 
to  $J$.
Equivalently,
\begin{equation}\la{gali2}
\dot X_m(t)=H X_m(t) + P_m N(X_m(t)).
\end{equation}

The Hamiltonian form guarantees the energy and charge conservation
 (\ref{EQ}):
\begin{equation}\la{EQ2}
E(X_m(t))=E(X_m(0)),\quad Q(X_m(t))=Q(X_m(0)),\qquad t\in\R.
\end{equation}
Indeed, the energy conservation holds by the Hamiltonian form  (\re{gal}), 
while the charge conservation holds
by
the Noether theory \ci{A,GSS87, KQ} due to the
$U(1)$-invariance of $E_m$, see (\re{U1}).

The equation (\re{gali2}) admits a unique local solution for every initial state
 $X_m(0)\in\cV_m$ since the right hand side 
is locally bounded and Lipschitz continuous. 
The global solutions exist by   (\re{EQV})
and
the energy and charge conservation (\re{EQ2}).
\medskip

Finally, we take any $X(0)\in\cV$ and choose a sequence 
\be\la{s0}
X_m(0)\to X(0),\qquad m\to\infty,
\ee
where the convergence holds in the metric of $\cV$.
Therefore, 
\be\la{EQm}
E(X_m(0))\to E(X(0)),\qquad Q(X_m(0))\to Q(X(0)).
\ee
Hence, (\ref{EQ2}) and (\re{EQV}) imply the basic uniform bound 
\be\la{Vb}
R:=\sup_{m\in\N}\,\,\sup_{t\in\R}| X_m(t)|_\cV <\infty.
\ee
Therefore, 
(\ref{gali2})  and  Proposition \re{p1} imply the second basic uniform bound
\be\la{Vb2}
\sup_{m\in\N}\,\sup_{t\in\R}\,\Vert \dot X_m(t)\Vert_{\cW^{-1}} <C(R),
\ee
since  the operator 
$H:\cW\to\cW^{-1}$ is bounded, and the projector
$P_m$ is also a bounded operator in $\cW\subset \cW^{-1}$.
Hence, 
the Galerkin
approximations $X_m(t)$ are uniformly Lipschitz-continuous with values in $\cV^{-1}$:
\be\la{ecg}
\sup_{m\in\N}\, d_{\cV^{-1}}(X_m(t),X_m(s))\le C(R)|t-s|,\qquad s,t\in\R.
\ee
Let us show that  
the uniform estimates   (\ref{Vb}) and (\ref{ecg}) imply a compactness of the 
Galerkin approximations and the conservation laws. 
Let us recall that $\cX:=\cV^0$ and $\cV:=\cV^1$.
\bp\la{lgal} Let   (\re{ro+}) hold and $X(0)\in\cV$. Then 
\medskip\\
i) 
There exists
a subsequence $m'\to\infty$ such that
\be\la{ss}
X_{m'}(t)\tocX X(t),\qquad m'\to\infty,\qquad t\in\R,
\ee
where $X(\cdot)\in C(\R, \cX)$.
\medskip\\
ii) Every limit function $X(\cdot)$
is a solution to  (\re{HS}), and $X(\cdot)\in C(\R,\cV)$.
\medskip\\
iii) The conservation laws (\re{EQ}) hold.
\ep
\begin{proof}
	i) The convergence (\re{ss}) follows from 
	(\re{Vb}) and  (\re{Vb2}) 
	by the Dubinsky
	`theorem on three spaces' \ci{Dub65}  (Theorem 5.1 of
	\ci{Lions}). Namely, the embedding $\cV\subset\cX$ is compact by the Sobolev theorem,
	and hence, (\re{ss}) holds by 
	(\re{Vb})
	for $t\in D$, where $D$ is a countable dense set. 
	Finally, let us use the  interpolation inequality  and  (\re{Vb}), (\re{ecg}):
	for any $\ve>0$
	\be\la{inti}
	d_\cX(X_m(t),X_m(s))\le\ve d_\cV(X_m(t),X_m(s))
	+ C(\ve)  d_{\cV^{-1}}(X_m(t),X_m(s))\le 2\ve R+C(\ve,R)|t-s|.
	\ee
	This inequality 
	implies the 
	equicontinuity of the Galerkin approximations with values in $\cX$. Hence,
	convergence 
	(\re{ss}) holds for all $t\in\R$ since it holds for 
	the dense set of $t\in D$. 
	The same equicontinuity also implies the continuity of the limit function $X\in C(\R, \cX)$.
	\medskip\\
	ii) 
	Integrating equation (\re{gali2}), we obtain 
	\begin{equation}\la{gal2}
	\int_0^t\langle \dot X_m(t),Y\rangle\,ds=\int_0^t\langle X_m(s),HY)\,ds + \int_0^t\langle N(X_m(s)),Y\rangle\,ds,
	\qquad Y\in\cW_m,
	\end{equation}
	Below we will write  $m$ instead of $m'$. 
	To prove  (\re{LPSiv}) it suffices to
	check that 
	in the limit $m\to\infty$, we get
	\begin{equation}\la{gal4}
	\int_0^t\langle \dot X(t),Y\rangle\,ds
	=\int_0^t(X(s),HY)\,ds + \int_0^t\langle N(X(s)),Y\rangle\,ds,
	\qquad Y\in\cW_n,\qquad n\in\N.
	\end{equation}
	The convergence of  the left hand side and of the first term on the right hand side 
	of (\re{gal2})
	follow
	from (\re{ss}) and  (\re{s0}) since $HY\in\cW_m$.

	It remains to consider the last integral in (\re{gal2}).
	The integrand is uniformly bounded by (\re{Vb}) and Proposition \re{p1}. 
	Hence, it suffices to check the pointwise
	convergence
	\be\la{Nm}
	\langle N(X_m(s), Y\rangle-\!\!\!-\!\!\!\!\to \langle N(X(s), Y\rangle,\qquad Y\in\cW_n
	\ee
	for any $s\in\R$. Here 
$N(X_m(s))=(ie\Phi^\otimes_m(s)\psi_m(s),p_m(s),f_m(s))$ according to 
	the notation (\re{HN}), and
	$Y=(\vp,\vka,v)\in \cW_n$. Hence, 
	(\re{Nm})
	reads
	\be\la{Nm2}
	ie[\Phi^\otimes_m(s)\psi_m(s),\vp]+p_m(s)\vka+f_m(s)v\,\to\,  
	ie(\Phi^\otimes(s)\5\psi(s),\vp)+p(s)\vka+f(s)v,
	\ee
	where $[\cdot,\cdot]$ is the scalar product in $L^2(\ov\T)$.
	The convergence of $p_m(s)\vka$ follows from (\re{ss}) (with $m'=m$) .
	To prove the convergence of the two remaining terms we first show that
	\be\la{2rt}
	\Phi_m(s):=G\rho_m \toLC \Phi(s):=G\rho.
	\ee
	Indeed, 
	 (\re{ss}) implies that 
	 \be\la{qp}
	 \psi_m(s)\toLtt\psi(s),\qquad q_m(s)\to q(s)
	\ee
	Further, the sequence $\psi_m(s)$ is bounded in $H^1(\ov\T)$ by (\re{Vb}). 
	Hence, the sequence $\rho_m(s)$ is bounded in 
	the Sobolev space
	$W^{1,3/2}(\T)$ by (\re{re4}).
	Therefore, the sequence $\rho_m(s)$ is precompact in $L^2(\T)$ by  the Sobolev compactness theorem. 
	Hence,
	\be\la{qp2}
	\rho_m\toLt\rho
	\ee
	by (\re{qp}).
	Therefore,
	(\re{2rt}) holds
	since the operator $G:L^2(\T)\to C(\T)$ is continuous. 
	Finally, (\re{2rt}) and (\re{qp}) imply that
	\be\la{Nm3}
	\Phi^\otimes_m(s)\psi_m(s)\toLtt\5 \Phi^\otimes(s)\psi(s),\qquad
	f_m(s)\to f(s),
	\ee
	which proves (\re{Nm2}). Now (\re{gal4}) is proved for $Y\in\cV_n$ with any $n\in\N$.
	Hence, $X(t)$ is a solution to (\re {HS}). Finally,  
	$N(X(t))$ is  bounded in $\cW$ by
  (\re{Vb}) and Proposition \re{p1}.
  Hence, 
 (\re{LPSiv}) implies that  $X(\cdot)\in C(\R,\cV)$.
	\medskip\\
	iii) The conservation laws (\re{EQ2}) and the convergences (\re{s0}), (\re{ss}) imply that
	\be\la{EQ3}
	E(X(t))\le E(X(0)),\quad Q(X(t))\le Q(X(0)),\qquad t\in\R.
	\ee
	The last inequality holds by the first convergence of (\re{qp}).
The first inequality follows from the representation
\begin{equation}\la{EQ4}
E(X_m(t))=\fr12 \Vert\na\psi_m(t)\Vert_{L^2({\T})}^2+\fr12 
\Vert \sqrt{G}\rho_m(t)\Vert_{L^2({\T})}^2 +
\sum_{n\in\Gamma_n}\fr{p_m^2(n,t)}{2M}.
\end{equation}
Namely, the last two terms on the right hand side converge 
by (\re{qp2}) and (\re{ss}). Moreover, 
the first term is bounded by (\re{Vb}). Hence,  the first convergence of
 (\re{qp}) implies the weak convergence 
 \begin{equation}\la{EQ5}
\na\psi_{m}(t)\toLwt\na\psi(t)
 \end{equation}
 by the Banach theorem.
Now the first inequality of (\re{EQ3}) follows by the property of the weak convergence
in the Hilbert space.
Finally, the opposite inequalities 
 to (\re{EQ3})
 are also true by the uniqueness of 
solutions $X(\cdot)\in C(\R,\cV)$, which is proved in Proposition \re{TLWP}.
\end{proof}



\begin{thebibliography}{99}

\bibitem{Adams} R.A. Adams, Sobolev Spaces,  Academic Press, NY, 1975.

\bibitem{A} V. Arnold,  Mathematical Methods of Classical Mechanics,
Springer, New York, 1978. 



\bibitem{BLR} F. Bonetto, J. L. Lebowitz,   L. Rey-Bellet,            
  Fourier's law: a challenge to theorists, p. 128-150 in:
Fokas, A. (ed.) et al., Mathematical physics 2000. International congress, London, GB, 2000, Imperial College Press, London, 2000.
      


\bibitem{Born}
 M. Born, K. Huang, Dynamical Theory of Crystal Lattices,
  The Clarendon Press, Oxford University Press, New York, 1998.





\bibitem{CLL2013}
E. Canc\`es, S. Lahbabi, M. Lewin,
Mean-field models for disordered crystals,
{\em J. Math. Pures Appl. (9)} {\bf 100} (2013), no. 2, 241-274. 

\bibitem{CS2012}
E. Canc\`es, G. Stoltz,
A mathematical formulation of the random phase approximation
for crystals,
{\em Ann. I. H. Poincar\'e - AN} {\bf 29} (2012), 887-925. 




\bibitem{CBL1998}
L. Catto, C. Le Bris, P.-L. Lions,
  The Mathematical Theory of Thermodynamic Limits: 
Thomas-Fermi Type Models,
Clarendon Press, Oxford, 1998.

\bibitem{CBL2001}
L. Catto, C. Le Bris, P.-L. Lions,
 On the thermodynamic limit for Hartree-Fock type models,
{\em Ann. Inst. Henri Poincar\'e, Anal. Non Lin\'eaire}
{\bf 18} (2001), no. 6, 687-760.

\bibitem{CBL2002}
L. Catto, C. Le Bris, P.-L. Lions,
On some periodic Hartree-type models for crystals,
{\em Ann. Inst. Henri Poincar\'e, Anal. Non Lin\'eaire}
{\bf 19} (2002), no. 2, 143-190.
                              
               
\bibitem{Dub65}
Yu. A.
Dubinsky, Weak convergence in non-linear elliptic and parabolic
equations, {\em Mat. USSR Sb.} {\bf 67} (4) (1965), 609-642 (in Russian).
                              

\bibitem{D1967}
F.J. Dyson, Ground-state energy of a finite system of charged particles,
 J. Math.
Phys. 8, 1538-1545 (1967). 

\bibitem{DL1968}
 F.J. Dyson, A. Lenard, 
Stability of matter I, {\em J. Math. Phys.} {\bf 8} (1967), 423-434; 
II, ibid. {\bf 9} (1968), 698-711. 


\bibitem{GV2005}
G. Giuliani, G. Vignale,
 Quantum Theory of the Electron Liquid,
 Cambridge University Press, Cambridge, 2005. 



 
 \bibitem{GSS87} M. Grillakis, J. Shatah, W. Strauss, 
 Stability theory of solitary waves in the presence of symmetry. I,
 {\em  J. Funct. Anal.} {\bf 74} (1987), 160-197.
 
  
  
\bibitem{Her1}
 L. H\"ormander, The Analysis of Linear Partial Differential Operators. I. Distribution Theory and Fourier Analysis, Springer, Berlin, 2003.

 
 
  
\bibitem{Kit}
C. Kittel, Introduction to Solid State Physics, Wiley \& Sons,  Hoboken, NJ, 2005.

  \bibitem{KQ}
  A. I. Komech, Quantum Mechanics: Genesis and Achievements, Springer, Dordrecht, 2013.
  
  
  

\bibitem{K2014}
A. I. Komech,
On crystal ground state in the Schr\"odinger--Poisson model, 
 {\em SIAM J. Math. Anal.} {\bf 47} (2015), no. 2, 1001-1021.
arXiv:1310.3084

\bibitem{K2015}
A. I. Komech,
On crystal ground state in the Schr\"odinger--Poisson model with point ions, 
submitted to {\em J. Math. Anal. Appl.}, 2015.  
arXiv:1409.1847



\bibitem{KKpl2015}
A. Komech, E. Kopylova,
 On the linear stability of crystals 
for the Schr\"odinger-Poisson model, {\em J. Stat. Phys.} {\bf 165} (2016),
no. 2, 246-273.  

arXiv:1505.07074






\bibitem{BL2005}	
C. Le Bris, P.-L. Lions,
From atoms to crystals: a mathematical journey,
{\em Bull. Am. Math. Soc., New Ser.}
{\bf  42} (2005), no. 3, 291-363.



\bibitem{LL1969}
J.L. Lebowitz, E.H. Lieb, Existence of thermodynamics for real matter with Coulomb
forces, {\em Phys. Rev. Lett.} {\bf 22} (13) (1969), 631-634.

\bibitem{LL1973}
J.L. Lebowitz, E.H. Lieb, 
Lectures on the thermodynamic limit for Coulomb systems, in:
Springer Lecture Notes in Physics, Vol. 20, Springer, 1973, pp. 136-161.





\bibitem{LS2014-1} 
M. Lewin, J. Sabin,
The Hartree equation for infinitely many particles. I. Well-posedness theory,  arXiv:1310.0603.

\bibitem{LS2014-2} 
M. Lewin, J. Sabin,
The Hartree equation for infinitely many particles. II. Dispersion and scattering in 2D,  arXiv:1310.0604.


\bibitem{LS2010}
 E.H. Lieb, R. Seiringer, The Stability of Matter in Quantum Mechanics,
Cambridge University Press, Cambridge, 2009.





\bibitem{Lions}
J.-L. Lions, Quelques m\'ethodes de r\'esolution des probl\`emes aux limites non lin\'eaires,
Dunod; Gauthier-Villars, Paris, 1969.





\bibitem{RS2} M. Reed, B. Simon, Methods of Modern Mathematical Physics, II: Fourier Analysis, Self-Adjointness, 
Academic Press, NY, 1975.

\bibitem{Stratton}
J.A. Stratton, Electromagnetic Theory, John Wiley \& Sons, Inc.,
Hoboken, New Jersey, 2007.


\bibitem{Zim}
M. Ziman, The Calculation of Bloch Functions, Academic Press, NY, 1971.





\end{thebibliography}
\end{document}